\documentclass[12]{amsart}
\usepackage{amsmath,amssymb,amsthm,color,enumerate,comment,centernot,enumitem,url,cite}
\usepackage{graphicx,relsize,bm}
\usepackage{mathtools}
\usepackage{array}

\makeatletter
\newcommand{\tpmod}[1]{{\@displayfalse\pmod{#1}}}
\makeatother

\newtheorem{thm}{Theorem}[section]
\newtheorem{lemma}[thm]{Lemma}

\theoremstyle{remark}

\theoremstyle{definition}

\newtheorem{rem}[thm]{Remark}

\theoremstyle{THM}

\newcommand{\abs}[1]{\left|{#1}\right|}

\def\FF {{\mathcal F}}
\def\GG {{\mathcal G}}

\def\Z {{\mathbb Z}}

\def\Q {{\mathbb Q}}

\def\GG {{\mathcal G}}

\def\Z {{\mathbb Z}}
\def\Q {{\mathbb Q}}

\def\Gal{{\mbox{{\rm{Gal}}}}}

\makeatletter
\@namedef{subjclassname@2020}{%
  \textup{2020} Mathematics Subject Classification}
\makeatother

\def\red#1 {\textcolor{red}{#1 }}
\def\blue#1 {\textcolor{blue}{#1 }}

\numberwithin{equation}{section}

\def\Z {{\mathbb Z}}

\newcommand{\mmod}[1]{\ \mathrm{mod}\enspace #1}
\begin{document}

\title[Monogenic Cyclic Trinomials of the Form $x^4+cx+d$]{Monogenic Cyclic Trinomials of the Form $x^4+cx+d$}

\author{Lenny Jones}
\address{Professor Emeritus, Department of Mathematics, Shippensburg University, Shippensburg, Pennsylvania 17257, USA}
\email[Lenny~Jones]{doctorlennyjones@gmail.com}

\date{\today}

\begin{abstract}
 A monic polynomial $f(x)\in \Z[x]$ of degree $n$ that is irreducible over $\Q$ is called \emph{cyclic} if the Galois group over $\Q$ of $f(x)$ is the cyclic group of order $n$, while $f(x)$ is called \emph{monogenic} if
 $\{1,\theta,\theta^2,\ldots, \theta^{n-1}\}$ is a basis for the ring of integers of $\Q(\theta)$, where $f(\theta)=0$. 
 In this article, we show that there do not exist any monogenic cyclic trinomials of the form $f(x)=x^4+cx+d$. 
This result, combined with previous work, proves that the only monogenic cyclic quartic trinomials are $x^4-4x^2+2$, $x^4+4x^2+2$ and $x^4-5x^2+5$.
\end{abstract}


\maketitle
\section{Introduction}\label{Section:Intro}

  We say that a monic polynomial $f(x)\in \Z[x]$ is \emph{monogenic} if $f(x)$ is irreducible over $\Q$ and $\{1,\theta,\theta^2,\ldots ,\theta^{\deg(f)-1}\}$ is a basis for the ring of integers $\Z_K$ of $K=\Q(\theta)$, where $f(\theta)=0$. Hence, $f(x)$ is monogenic if and only if  $\Z_K=\Z[\theta]$. For the minimal polynomial $f(x)$ of an algebraic integer $\theta$ over $\Q$, it is well known \cite{Cohen} that
\begin{equation} \label{Eq:Dis-Dis}
\Delta(f)=\left[\Z_K:\Z[\theta]\right]^2\Delta(K),
\end{equation}
where $\Delta(K)$ is the discriminant over $\Q$ of the number field $K$.
Thus, from \eqref{Eq:Dis-Dis}, $f(x)$ is monogenic if and only if $\Delta(f)=\Delta(K)$.

In a private communication, Tristan Phillips asked the author if it is possible to determine all monogenic cyclic quartic trinomials; that is, monogenic quartic trinomials that have Galois group isomorphic to the cyclic group of order 4. Recently, a partial answer to the question of Phillips was given in two separate papers \cite{JonesBAMSEven} and \cite{HJF1}. In \cite{JonesBAMSEven}, the author showed that the only monogenic cyclic trinomials of the form $x^4+bx^2+d\in \Z[x]$ are
\begin{equation}\label{Eq:Complete}
x^4-4x^2+2,\quad x^4+4x^2+2\quad \mbox{and} \quad x^4-5x^2+5,
\end{equation} while in \cite{HJF1}, Joshua Harrington and the author proved that there do not exist any monogenic cyclic trinomials of the form $x^4+ax^3+d\in \Z[x]$. In this article, we show that the complete solution to Phillips' question is precisely the set of trinomials in \eqref{Eq:Complete} by establishing the following result.
\begin{thm}\label{Thm:Main}
  There do not exist any monogenic cyclic quartic trinomials of the form $x^4+cx+d\in \Z[x]$.
\end{thm}

\section{Preliminaries}\label{Section:Prelim}
Throughout this article, we let
\begin{equation}\label{Eq:fandr}
f(x):=x^4+cx+d \quad \mbox{and}\quad r(x):=x^3-4dx-c^2,
\end{equation}
where $c,d\in \Z$ with $cd\ne 0$. 
 Straightforward calculations in Maple reveal that
\begin{equation}\label{Eq:Deltafandr}
\Delta(f)=\Delta(r)=256d^3-27c^4.
\end{equation}
\begin{rem}
The polynomial $r(x)$ in \eqref{Eq:fandr} is known as \emph{the cubic resolvent of $f(x)$}.
\end{rem}

The first theorem in this section follows from a result due to Jakhar, Khanduja and Sangwan \cite[Theorem 1.1]{JKS2} for arbitrary irreducible trinomials when applied to our specific quartic trinomial $f(x)$. Note that, for a prime $q$, we use 
the notation $q^N\mid \mid M$ to mean that $q^N$ is the exact power of $q$ that divides the integer $M$.
\begin{thm}\label{Thm:JKS}
Let $f(x)$ be irreducible over $\Q$. Let $K=\Q(\theta)$, where $f(\theta)=0$, and let $\Z_K$ be the ring of integers of $K$. A prime factor $q$ of $\Delta(f)$ does not divide $\left[\Z_K:\Z[\theta]\right]$ if and only if $q$ satisfies one of the following conditions:
\begin{enumerate}[font=\normalfont]
  \item \label{JKS:I1} when $q\mid c$ and $q\mid d$, then $q^2\nmid d$;\\
  \item \label{JKS:I2} when $q\mid c$ and $q\nmid d$, then
  \[\mbox{either} \quad q\mid c_2 \mbox{ and } q\nmid d_1 \quad \mbox{ or } \quad q\nmid c_2\left(dc_2^{4}+d_1^{4}\right),\]
  where $c_2=c/q$, and $d_1=\frac{d+(-d)^{q^j}}{q}$ with $q^j\mid\mid 4$;\\
  \item \label{JKS:I3} when $q\nmid c$ and $q\mid d$, then
  \[\mbox{either} \quad q\mid c_1 \mbox{ and } q\nmid d_2 \quad \mbox{ or } \quad q\nmid c_1\left(cc_1^3-d_2^3\right),\]
  where $c_1=\frac{c+(-c)^{q^\ell}}{q}$ with $q^\ell\mid\mid 3$, and $d_2=d/q$;\\
  \item \label{JKS:I4} when $q\nmid cd$, then $q^2\nmid \left(256d^3-27c^4\right)$.
   \end{enumerate}
\end{thm}

The next theorem follows from a result due to Kappe and Warren \cite[Theorem 1]{KW} when applied to our specific quartic trinomial $f(x)$.
\begin{thm}\label{Thm:KW}
Let $f(x)$ and $r(x)$ be as defined in \eqref{Eq:fandr}.
Suppose that $f(x)$ is irreducible over $\Q$, and that $r(x)$ has exactly one root $t\in \Z$. Let $L$ be the splitting field of $r(x)$ over $\Q$, and define
 \begin{equation}\label{Eq:gr}
 g(x):=(x^2-tx+d)(x^2-t).
 \end{equation} Then $\Gal(f)\simeq$
 \begin{enumerate}
   \item $C_4$ if and only if $g(x)$ splits over $L$,
   \item $D_4$ if and only if $g(x)$ does not split over $L$.
 \end{enumerate}
 \end{thm}
 
\section{The Proof of Theorem \ref{Thm:Main}}\label{Section:MainProof}
We first prove two lemmas.
\begin{lemma}\label{Lem:Quadratic fields}
  Let $A,B\in \Z$ such that neither $A$ nor $B$ is a square. Then $AB$ is a square if and only if $\Q(\sqrt{A})=\Q(\sqrt{B})$.
\end{lemma}
\begin{proof}
  Since $A$ and $B$ are not squares, we can write
  \[A=u^2a \quad \mbox{and} \quad B=v^2b,\]
  where $a$ and $b$ are squarefree integers with $a,b\ne 1$.
   Then $AB$ is a square if and only if $a=b$ if and only if \[\Q(\sqrt{A})=\Q(\sqrt{a})=\Q(\sqrt{b})=\Q(\sqrt{B}).\qedhere\]
\end{proof}
The next lemma is a more user-friendly version of Theorem \ref{Thm:KW}.
\begin{lemma}\label{Lem:KW}
 Let $f(x)$ and $r(x)$ be as defined in \eqref{Eq:fandr}. Suppose that $f(x)$ is irreducible over $\Q$, and that $r(x)$ has exactly one root $t\in \Z$.
 Define
 \[\delta_1:=t(16d-3t^2) \quad \mbox{and}\quad \delta_2:=(t^2-4d)(16d-3t^2).\]
   Then $\Gal(f)\simeq$
 \begin{enumerate}
   \item $C_4$ if and only if $\delta_1$ and $\delta_2$ are nonzero squares in $\Z$,
   \item $D_4$ if and only if neither $\delta_1$ nor $\delta_2$ is a nonzero square in $\Z$.
 \end{enumerate}
\end{lemma}
\begin{proof}
 From Theorem \ref{Thm:KW}, we have that
\begin{equation}\label{Eq:r2}
r(x)=(x-t)(x^2+Ax+B)=x^3+(A-t)x^2+(B-tA)x-tB,
\end{equation} for some $t,A,B\in \Z$, where $x^2+Ax+B$ is irreducible over $\Q$. Equating coefficients on $r(x)$ in \eqref{Eq:fandr} and \eqref{Eq:r2} yields
\begin{equation}\label{Eq:r3}
 r(x)=(x-t)(x^2+tx+t^2-4d).
\end{equation} Then, calculating $\Delta(r)$  in  \eqref{Eq:r3}, and recalling $\Delta(r)$ from \eqref{Eq:Deltafandr}, gives us 
\begin{equation}\label{Eq:Delta(f)}
\Delta(f)=\Delta(r)=256d^3-27c^4=(16d-3t^2)(3t^2-4d)^2.
\end{equation} We also see from \eqref{Eq:r3} that $16d-3t^2$ is not a square in $\Z$ since $x^2+tx+t^2-4d$ is irreducible over $\Q$, so that
\[L:=\Q(\sqrt{16d-3t^2}) \ \mbox{is the splitting field of $r(x)$ over $\Q$}\] with $[L:\Q]=2$.
Furthermore, from \eqref{Eq:fandr}, we have that
\begin{equation}\label{Eq:r(t)=0}
t(t^2-4d)=c^2
\end{equation}
since $r(t)=0$. Thus, we deduce from \eqref{Eq:r(t)=0} that exactly one of the following sets of conditions holds:
\begin{align}
&\label{C1} \mbox{$t$ and $t^2-4d$ are both squares in $\Z$},\\
&\label{C2} \mbox{neither $t$ nor $t^2-4d$ is a square in $\Z$}.
\end{align}
Note that $t(t^2-4d)\ne 0$ since $c\ne 0$.
Suppose that \eqref{C1} holds, and let $2^{2k}\mid\mid t$. Then, from \eqref{Eq:r(t)=0}, if $k=0$, then $2\nmid c$ and, since $\sqrt{t}\mid c$, we have that $t\pm c/\sqrt{t}$ are even integers. If $k\ge 1$, then $2^k\mid\mid \sqrt{t}$ and $2^{k+1}\mid\mid c$. Hence, again $t\pm c/\sqrt{t}$ are even integers. Thus, in any case, we deduce that
$(t^2\pm c\sqrt{t})/(2t)\in \Z$, so that
\begin{align*}
x^4+cx+d&=x^4+cx+\frac{t^3-c^2}{4t}\\
&=\left(x^2+\sqrt{t}x+\frac{t^2-c\sqrt{t}}{2t}\right)\left(x^2+\sqrt{t}x+\frac{t^2+c\sqrt{t}}{2t}\right)
\end{align*}
is a factorization of $f(x)$ in $\Z[x]$, contradicting the fact that $f(x)$ is irreducible over $\Q$. Hence, \eqref{C2} holds.
Then, from \eqref{Eq:r(t)=0}, we see that
\begin{equation}\label{Eq:delta1delta2}
\delta_1\delta_2=t(16d-3t^2)(t^2-4d)(16d-3t^2)=c^2(16d-3t^2)^2
\end{equation} is a nonzero square in $\Z$. Therefore, exactly one of the following sets of conditions holds:
\begin{align}
&\label{C3} \mbox{$\delta_1$ and $\delta_2$ are both squares in $\Z$},\\
&\label{C4} \mbox{neither $\delta_1$ nor $\delta_2$ is a square in $\Z$}.
\end{align}

Let $K$ be the splitting field of $f(x)$. Then, it follows from the details of the proof of \cite[Theorem 1]{KW} that the following quadratic fields are subfields of $K$:
\[\begin{array}{lcl}
L=\Q(\sqrt{16d-3t^2}), &\quad & M_3=\Q(\sqrt{t(16d-3t^2)})=\Q(\sqrt{\delta_1})\\[.4em]
M_1=\Q(\sqrt{t}), & & M_4=\Q(\sqrt{(t^2-4d)(16d-3t^2)})=\Q(\sqrt{\delta_2}),\\[.4em]
M_2=\Q(\sqrt{t^2-4d}). &&
\end{array}\]
 Thus, by \eqref{Eq:r(t)=0}, \eqref{C2} and Lemma \ref{Lem:Quadratic fields}, we deduce that $M_1=M_2$. 

 Suppose that \eqref{C4} holds. Then, since $16d-3t^2$ is not a square, it follows that $M_3=M_4$ and $M_1\ne L$ by \eqref{C2}, \eqref{Eq:delta1delta2} and Lemma \ref{Lem:Quadratic fields}. Hence, $K$ contains more than a single quadratic subfield, which implies that $\Gal(f)\simeq D_4$.

 Conversely, suppose that $\Gal(f)\simeq D_4$ and, by way of contradiction, assume that \eqref{C3} holds. We let $\alpha_1,\alpha_2,\alpha_3,\alpha_4$ be the roots of $f(x)$ and follow the proof of \cite[Theorem 1.]{KW}. Then, the quadratic subfield $L$ of $K$ contains:
 \begin{align*}
 \alpha_1\alpha_2+\alpha_3\alpha_4, \ \alpha_1\alpha_3+\alpha_2\alpha_4, \ \alpha_1\alpha_4+\alpha_2\alpha_3& \quad \mbox{(the roots of $r(x)$ in \eqref{Eq:r3})},\\
 \alpha_1\alpha_2, \ \alpha_3\alpha_4, \ \alpha_1+\alpha_2, \ \alpha_3+\alpha_4 & \quad \mbox{(the roots of $g(x)$ in \eqref{Eq:gr}).}
 \end{align*}
 \noindent
 Consider the polynomial
 \[h(x):=x^2-(\alpha_1+\alpha_2)x+\alpha_1\alpha_2\in L[x].\] Observe that the roots of $h(x)$ are $\alpha_1$ and $\alpha_2$. Since $\alpha_1\ne \alpha_2$ and
    \[\alpha_3-\alpha_4=\frac{(\alpha_1\alpha_3+\alpha_2\alpha_4)-(\alpha_1\alpha_4+\alpha_2\alpha_3)}{\alpha_1-\alpha_2}\in L(\alpha_1),\]
  it follows that $\alpha_3, \alpha_4\in L(\alpha_1)$ so that $K=L(\alpha_1)$, which yields the contradiction that $\abs{\Gal(f)}=4$, and completes the proof of the lemma.
 \end{proof}
\begin{proof}[Proof of Theorem \ref{Thm:Main}]
By way of contradiction, suppose that $f(x)$ is a monogenic cyclic trinomial.
Since $f(x)$ is cyclic, we must have $\Delta(f)>0$, which implies that
\begin{equation}\label{Eq:Positive}
d>0 \quad \mbox{and} \quad 16d-3t^2>0,
\end{equation}
from \eqref{Eq:Delta(f)}. Furthermore, we have from Lemma \ref{Lem:KW} that
$\delta_1$ and $\delta_2$ are nonzero squares in $\Z$. Hence, if $t<0$, then 
$\delta_1=t(16d-3t^2)<0$, contradicting the fact that $\delta_1$ is a square. Thus, $t>0$ and $t^2-4d>0$ by \eqref{Eq:r(t)=0}. Therefore, it follows that
\begin{equation*}\label{Eq:Negative}
(3t^2-4d)-(16d-3t^2)=6t^2-20d>6t^2-24d=6(t^2-4d)>0,
\end{equation*}
which implies that
\begin{equation}\label{Eq:Size}
3t^2-4d>16d-3t^2>0
\end{equation} from \eqref{Eq:Positive}.
We make the following additional observations:
\begin{align}\label{Obs}
\begin{split} 
& 2\mid \Delta(f)  \ \Longleftrightarrow \ 4\mid (3t^2-4d) \ \Longleftrightarrow  \ 4\mid (16d-3t^2) \Longleftrightarrow \ 2\mid t \\
& 3\mid \Delta(f) \ \Longleftrightarrow \ 3\mid (3t^2-4d) \ \Longleftrightarrow  \ 3\mid (16d-3t^2) \Longleftrightarrow \ 3\mid d.
\end{split}
 \end{align}

In our arguments, it will be useful to know the solutions to the equation
\begin{equation}\label{Eq:2^k}
3t^2-4d=2^{k-1},
 \end{equation} where $k$ is an integer with $k\ge 1$.
Solving \eqref{Eq:2^k} for $4d$, and using \eqref{Eq:r(t)=0}, yields the equation
\begin{equation}\label{Eq:PreElliptic}
c^2=-2t^3+2^{k-1}t.
\end{equation}
Then, multiplying both sides of \eqref{Eq:PreElliptic} by $4$ gives rise to the elliptic curve:
\begin{align}\label{E}
\begin{split}
E_k: & \quad Y^2=X^3-2^kX,\\
\mbox{where } & \mbox{$X=-2t$ and $Y=2c$.}
\end{split}
\end{align}
Note that the point $(0,0)$ is on $E_k$. However, since $c\ne 0$, we must have $Y\ne 0$.
Thus, it follows from \cite{Dra,Walsh} that the set $S$ of integral points $(X,Y)$ on $E_k$ with $Y>0$ is
\begin{equation}\label{Eq:S}
  S=\left\{\begin{array}{cl}
  \varnothing & \mbox{if $\overline{k}\in \{0,2\}$}\\[.5em] 
  \{(-1,1), (2,2), (2\cdot 13^2,2\cdot 13\cdot 239)\}& \mbox{if $k=1$}\\[.5em]
  \{(-2^{(k-1)/2},2^{3(k-1)/4}),(2^{(k+1)/2},2^{(3k+1)/4}), &\\[.5em]
  (2^{(k+1)/2}\cdot 13^2,2^{(3k+1)/4}\cdot 13\cdot 239),&\mbox{if $k\ge 5$ with $\overline{k}=1$}\\[.5em]
  (2^{(k-5)/2}\cdot 3^2,2^{3(k-5)/4}\cdot 3 \cdot 7)\}& \\[.5em] 
  \varnothing & \mbox{if $\overline{k}=3$,}
    \end{array}\right.
\end{equation}
where $\overline{k}:=k\mmod{4}\in \{0,1,2,3\}$. Therefore, every viable integer coefficient pair $(c,d)$ for $f(x)$ arises from an integer solution $(d,t,k)$ of \eqref{Eq:2^k}, which in turn, arises from an integral point $(X,Y)$ in \eqref{Eq:S}. Note, however, that an integral point $(X,Y)$ in \eqref{Eq:S} does not always yield an integer solution $(d,t,k)$ of \eqref{Eq:2^k}, or a viable integer coefficient pair $(c,d)$ for $f(x)$. For example,
 if $k=1$, then we deduce from \eqref{Eq:S} that any integer solutions $(d,t)$ to \eqref{Eq:2^k} must arise from the integral points
 \begin{equation}\label{Eq:k=1}
 (X,Y)\in \{(2,2), (2\cdot 13^2,2\cdot 13\cdot 239)\}
 \end{equation}
 on $E_k$. That is, the integer pairs $(c,t)$ corresponding to the integral points $(X,Y)$ in \eqref{Eq:k=1} are precisely
 \[(c,t)\in \{(\pm 1,-1),(\pm 13\cdot 239,-13^2),\] which in turn, correspond precisely to the coefficient pairs 
 \begin{equation}\label{Eq:coeffk=1}
 (c,d)\in \{(\pm 1,1/2), (\pm 13\cdot 239,42841/2).
  \end{equation} Since $d\not \in \Z$ in all pairs of \eqref{Eq:coeffk=1},  it follows that there are no monogenic cyclic quartic trinomials $f(x)\in \Z[x]$ when $k=1$ in \eqref{Eq:2^k}.
  Therefore, we can assume that $3t^2-4d>1$. Moreover, we claim that
  \begin{equation}\label{Eq:Claim1}
    3t^2-4d\ne 2^{k-1} \ \mbox{for any integer $k\ge 1$.}
  \end{equation}
  Since the case $k=1$ has been addressed, in order to establish \eqref{Eq:Claim1}, we assume that $3t^2-4d=2^{k-1}$ for some integer $k\ge 2$, and we proceed by way of contradiction.
Observe then that $2\mid t$ which implies that $8\mid t(t^2-4d)$. It follows from  \eqref{Eq:r(t)=0} that $16\mid t(t^2-4d)$ and $4\mid c$.  Consequently, we must address the following three cases:
\begin{enumerate}
  \item \label{Case 1} $4\mid t$ and $2\mid d$,
  \item \label{Case 2} $4\mid t$ and $2\nmid d$,
  \item \label{Case 3} $2\mid \mid  t$ and $2\nmid d$.
\end{enumerate}
Note that if $2\mid \mid  t$ and $2\mid d$, then $2^3\mid \mid t(t^2-4d)$, which contradicts \eqref{Eq:r(t)=0}, and so this situation is not possible.

For case \eqref{Case 1}, if $4\mid d$, then condition \eqref{JKS:I1} of Theorem \ref{Thm:JKS} fails for the prime $q=2$. Hence, $2\mid \mid d$ and
$k=4$. It then follows from \eqref{Eq:S} that there are no cyclic monogenic trinomials $f(x)$ in this case.

Suppose next that case \eqref{Case 2} holds. It is then easy to see that $k=3$, and again we conclude from \eqref{Eq:S} that there are no cyclic monogenic trinomials $f(x)$ in this case as well.

Finally, suppose that case \eqref{Case 3} holds. Then $2^2\mid \mid (-2t)$, and therefore, from \eqref{Eq:S}, we get the following viable integral points on the elliptic curves in \eqref{E}:
\begin{equation*}\label{Eq:Viable}
(X,Y)=\left\{\begin{array}{cl}
    (-4,\pm 8) & \mbox{if $k=5$}\\[.5em] 
  (36,\pm 168) & \mbox{if $k=9$}, 
\end{array}\right.
\end{equation*}
which in turn, using \eqref{Eq:r(t)=0}, yield the corresponding integer triples
\begin{equation}\label{Eq:triples}
 (t,c^2,d)=\left\{\begin{array}{cl}
  (2,16,-1) & \mbox{if $k=5$}\\[.5em]
(-18,7056,179) & \mbox{if $k=9$.} 
\end{array}\right.
\end{equation} Since we have shown that both $t$ and $d$ must be positive, we conclude that neither of the triples in \eqref{Eq:triples} produces a cyclic monogenic trinomial $f(x)$, which establishes the claim in \eqref{Eq:Claim1}.
Consequently, we can assume that $q\mid (3t^2-4d)$ for some prime $q\ge 3$.

Next, we claim that for any prime $q$,
\begin{equation}\label{Eq:Claim2}
 q\mid (3t^2-4d) \quad \mbox{if and only if} \quad  q\mid (16d-3t^2).
\end{equation}
Note by \eqref{Obs} that we only need to address primes $q\ge 5$.

Suppose first that $q\ge 5$ is a prime such that $q\mid (3t^2-4d)$ and $q\nmid (16d-3t^2)$. Since $3t^2\equiv 4d \pmod{q}$, we have that
\[16d-3t^2\equiv 12d\not \equiv 0 \pmod{q},\] which implies that $q\nmid d$. If $q\mid c$, then we see from \eqref{Eq:Delta(f)} that  $q\mid 256d^3$, since $q\mid \Delta(f)$. Thus, since $q\ne 2$, we deduce that $q\mid d$, which is a contradiction. Hence, $q\nmid c$. Therefore, we have that $q\nmid cd$. However, since $q\mid (3t^2-4d)$, we conclude from \eqref{Eq:Delta(f)} that $q^2\mid (256d^3-27c^4)$, which implies that condition \eqref{JKS:I4} of Theorem \ref{Thm:JKS} fails. Hence, $f(x)$ is not monogenic.

Suppose next that $q\ge 5$ is a prime such that $q\mid (16d-3t^2)$ and $q\nmid (3t^2-4d)$. Since $3t^2\equiv 16d \pmod{q}$, we have that
\[3t^2-4d\equiv 12d\not \equiv 0 \pmod{q},\] which implies that $q\nmid d$. If $q\mid t$, then 
\[16d=(16d-3t^2)+3t^2\equiv 0 \pmod{q},\] which yields the contradiction that $q\mid d$ since $q\ge 5$. Similarly, if $q\mid (t^2-4d)$, then 
\[4d=(16d-3t^2)+3(t^2-4d)\equiv 0 \pmod{q},\] which again yields the contradiction that $q\mid d$. Thus, $q\nmid t$ and $q\nmid(t^2-4d)$ so that 
$q\nmid c$ by \eqref{Eq:r(t)=0}. Hence, $q^2\nmid (16d-3t^2)$ since otherwise, condition \eqref{JKS:I4} of Theorem \ref{Thm:JKS} fails. That is, $q\mid\mid (16d-3t^2)$, which implies that $q\mid \mid t(16d-3t^2)$ contradicting the fact that $t(16d-3t^2)$ is a square. Therefore, the claim in \eqref{Eq:Claim2} is established.

Suppose now that $q$ is a prime divisor of $3t^2-4d$, and consequently of $16d-3t^2$ by \eqref{Eq:Claim2}. We claim that 
\begin{equation}\label{Eq:Claim3}
 \mbox{if $q\ge 5$ then $q\mid \mid (3t^2-4d)$ and $q\mid \mid (16d-3t^2)$.}
\end{equation} 
 We only give details to show that $q\mid \mid (3t^2-4d)$ since the details are identical to show that $q\mid \mid (16d-3t^2)$. By way of contradiction, suppose that $q^2\mid(3t^2-4d)$. If $q\mid d$, then, since $q\ne 3$, we see that $q\mid t$, and so $q\mid c$ from \eqref{Eq:r(t)=0}. Thus, $q^2\mid 4d$, so that $q^2\mid d$ since $q\ne 2$. Hence, it follows that $f(x)$ is not monogenic since condition \eqref{JKS:I1} of Theorem \ref{Thm:JKS} fails. Therefore, $q\nmid d$. Note then from \eqref{Eq:Delta(f)}, that $q\nmid c$ since $q\mid \Delta(f)$ and $q\ne 2$. That is, $q\nmid cd$. Consequently, since $q^2\mid \Delta(f)$, it follows that $f(x)$ is not monogenic since condition \eqref{JKS:I4} of Theorem \ref{Thm:JKS} fails, and this final contradiction completes the proof of the claim in \eqref{Eq:Claim3}.

Thus, we can therefore assume that
\begin{equation}\label{Eq:Factors}
3t^2-4d=2^a3^b\prod_{i=1}^{m}q_i \quad \mbox{and} \quad 16d-3t^2=2^u3^v\prod_{i=1}^{m}q_i
\end{equation}
for some nonnegative integers $a,b,u,v,m$, such that all of the following conditions hold:
\begin{enumerate}[label=(\roman*)]
\item \label{Ii} $a+b>0$ and $u+v>0$,\\
 (since otherwise $4d-3t^2=16d-3t^2$, which is impossible because $d>0$),
  \item \label{Iii} $a>0$ if and only if $u>0$,
  \item \label{Iiii} $b>0$ if and only if $v>0$,
  \item \label{Iiv} $2^a3^b>2^u3^v$ (by \eqref{Eq:Size}),
  \item \label{Iv} the $q_i$ are primes with $q_i\ge 5$.
  \end{enumerate} 
 We claim that
 \begin{equation}\label{Eq:Claim4}
   \mbox{$b\ne v$ in \eqref{Eq:Factors}.}
 \end{equation}
  To establish \eqref{Eq:Claim4}, we assume, by way of contradiction, that $b=v$ in \eqref{Eq:Factors}. 
  Thus, we have from \eqref{Eq:Factors} that
  \begin{equation}\label{Eq:1}
  \frac{3t^2-4d}{16d-3t^2}=2^{a-u}\in \Z,
  \end{equation} since $2^a>2^u$ by \ref{Iiv}. Note then that $u\ge 2$ by \eqref{Obs}. Since
  \begin{equation*}\label{Eq:Main}
  \frac{3t^2-4d}{16d-3t^2}=\frac{3t^2-16d+12d}{16d-3t^2}=-1+\frac{12d}{16d-3t^2},
  \end{equation*}
  we can rewrite the equation in \eqref{Eq:1} as
  \begin{equation*}
    (2^{a-u}+1)(16d-3t^2)=12d,
  \end{equation*}
  which implies that $u=\nu_2(d)+2$, where $\nu_2(*)$ is the exact power of 2 dividing $*$. Since the equation in \eqref{Eq:1} can also be rewritten as
  \begin{equation*}
    3t^2(2^{a-u}+1)=4d(2^{a-u+2}+1),
  \end{equation*} it follows that $\nu_2(t^2)=\nu_2(d)+2=u$, so that $2\mid u$. Therefore, we can write
  \[t^2=2^ut_1^2 \quad \mbox{and}\quad d=2^{u-2}d_1,\]
  where $2\nmid t_1d_1$. Thus, 
  \begin{equation}\label{Eq:2}
  \frac{3t^2-4d}{16d-3t^2}=\frac{2^u3t_1^2-2^ud_1}{2^{u+2}d_1-2^u3t_1^2}=\frac{3t_1^2-d_1}{4d_1-3t_1^2}=2^{a-u}.
   \end{equation} If $4\mid (3t_1^2-d_1)$, then $d\equiv 3 \pmod{3}$, so that $t_1^2-d_1\equiv 2 \pmod{4}$. Then,
  \[t(t^2-4d)=2^{u/2}t_1(2^ut_1^2-2^ud_1)=2^{(3u+2)/2}t_1\left(\frac{t_1^2-d_1}{2}\right)\in \Z,\] and we deduce from \eqref{Eq:r(t)=0} that $2\mid (3u+2)/2$. Consequently, $u\equiv 2 \pmod{4}$ so that $2\nmid 3u/2$. However, we then arrive at the contradiction that
  \[\delta_1=t(16d-3t^2)=2^{u/2}t_1(2^{u+2}d_1-2^u3t_1^2)=2^{3u/2}t_1(4d_1-3t_1^2)\] is not a square since $2^{3u/2}\mid \mid \delta_1$. Therefore, $4\nmid (3t_1^2-d_1)$, and we conclude from \eqref{Eq:2} that $a-u=1$ in \eqref{Eq:1}. Thus,
  \[3t^2-4d-2(16d-3t^2)=9(t^2-4d)=0,\] which, from \eqref{Eq:r(t)=0}, yields the contradiction that $c=0$ and establishes the claim in \eqref{Eq:Claim4}.

  Hence, since $b\ne v$, we may assume by \ref{Iiii} that either $b>v>0$ or $v>b>0$, which implies that either $b\ge 2$ or $v\ge 2$. Thus, $3\mid \Delta(f)$ and $3\mid d$ by \eqref{Obs}. We use Theorem \ref{Thm:JKS} with $q=3$ to show that  $3\mid \left[\Z_K:\Z[\theta]\right]$, where $K=\Q(\theta)$ and $f(\theta)=0$. Since $3\mid d$, we only have to examine conditions \eqref{JKS:I1} and \eqref{JKS:I3} of Theorem \ref{Thm:JKS}.

  Suppose first that $3\mid c$. Then $3\mid t$ from \eqref{Eq:r(t)=0} so that $9\mid t^2$. Hence, if $b\ge 2$, then $9\mid (3t^2-4d)$, which implies that $9\mid d$. Similarly, if $v\ge 2$, then $9\mid (16d-3t^2)$, which implies that $9\mid d$. In either case, condition \eqref{JKS:I1} of Theorem \ref{Thm:JKS} fails and $f(x)$ is not monogenic.

  Suppose next that $3\nmid c$. If $9\mid d$ and $b\ge 2$, then
  \[3t^2-4d\equiv 3t^2\equiv 0 \pmod{9},\] which implies that $3\mid t$, yielding the contradiction that $3\mid c$ by \eqref{Eq:r(t)=0}. The same contradiction is reached if $v\ge 2$. Thus, $9\nmid d$.  
  With these restrictions on $c$ and $d$, we arrive at the set of pairs 
  \[P:=\{(c \mmod{9},d \mmod{9})\}=\{(2,6),(4,3),(5,3),(7,6)\}\]
  for which $t(t^2-4d)\equiv c^2\pmod{9}$ (from \eqref{Eq:r(t)=0}). With $c_1=(c-c^3)/3$ and $d_2=d/3$, it is easy to see that
     \[cc_1^3-d_2^3\equiv 0 \pmod{3} \quad \mbox{if and only if} \quad c^4-c^2+d\equiv 0 \pmod{9}.\] Straightforward calculations reveal that $9\nmid 3c_1$ and $c^4-c^2+d \equiv 0 \pmod{9}$ for each of the pairs in $P$. Hence, condition \eqref{JKS:I3} of Theorem \ref{Thm:JKS} fails, and $f(x)$ is not monogenic, which completes the proof of the theorem.
\end{proof}






\begin{thebibliography}{99}

\bibitem{Cohen} H. Cohen, \emph{A Course in Computational Algebraic Number Theory}, {Springer-Verlag}, 2000.

\bibitem{Dra} K. A. Draziotis, \emph{Integer points on the curve $Y^2=X^3\pm p^kX$}, Math. Comp. {\bf 75} (2006), no. 255, 1493--1505.

\bibitem{HJF1} J. Harrington and L. Jones, \emph{Monogenic trinomials of the form $x^4+ax^3+d$ and their Galois groups},  J. Algebra Appl. (to appear). 

\bibitem{JKS2} A. Jakhar, S. Khanduja and N. Sangwan, \emph{Characterization of primes dividing the index of a trinomial}, Int. J. Number Theory {\bf 13} (2017), no. 10, 2505--2514.

\bibitem{JonesBAMSEven} L. Jones, \emph{Monogenic even quartic trinomials}, Bull. Aust. Math. Soc. (to appear).

\bibitem{KW}  L-C. Kappe and B. Warren, \emph{An elementary test for the Galois group of a quartic polynomial}, Amer. Math. Monthly {\bf 96} (1989), no. 2, 133--137.

\bibitem{Walsh}  P. G. Walsh, \emph{Integer solutions to the equation $y^2=x(x^2\pm p^k)$}, Rocky Mountain J. Math. {\bf 38} (2008), no. 4, 1285--1302.

\end{thebibliography}
\end{document}